\documentclass[a4paper,12pt,reqno]{amsart}
\usepackage{amssymb,amsmath,amsthm}

\theoremstyle{plain}
\newtheorem{Theorem}{Theorem}[section]
\newtheorem{theorem0}{Thorem}
\newtheorem{Proposition}[Theorem]{Proposition}
\newtheorem{Proposition-Definition}[Theorem]{Proposition-Definition}
\newtheorem{Lemma}[Theorem]{Lemma}

\theoremstyle{definition} 
\newtheorem{Definition}[Theorem]{Definition}
\newtheorem{Remark}[Theorem]{Remark}
\newtheorem{Example}[Theorem]{Example}

\newtheorem{Observation}[Theorem]{Observation}

\def\A{{\mathcal{A}}}

\def\M{{\mathcal{M}}}
\def\Z{{\mathbb{Z}}}
\def\N{{\mathbb{N}}}
\def\Q{{\mathbb{Q}}}

\def\x{{\mathbf{x}}}
\def\y{{\mathbf{y}}}
\def\z{{\mathbf{z}}}

\def\a{{\mathbf{a}}}
\def\b{{\mathbf{b}}}
\def\c{{\mathbf{c}}}
\def\e{{\mathbf{e}}}
\def\m{{\mathbf{m}}}

\def\initial{\mathop{\mathrm{in}}\nolimits}
\def\min{\mathop{\mathrm{min}}\nolimits}
\def\mv{\mathop{\mathrm{mv}}\nolimits}

\def\max{\mathop{\mathrm{max}}\nolimits}
\def\Ker{\mathop{\mathrm{Ker}}\nolimits}
\def\deg{\mathop{\mathrm{deg}}\nolimits}

\def\initial{\mathop{\mathrm{in}}\nolimits}
\def\reg{\mathop{\mathrm{reg}}\nolimits}

\providecommand{\abs}[1]{\lvert#1\rvert}
\title{Toric ideals for high Veronese subrings of toric algebras}
\author{Takafumi Shibuta}
\address{
\begin{flushleft}
	\hspace{0.3cm} Department of Mathematics, Rikkyo University, Nishi-Ikebukuro, Tokyo 171-8501, Japan\\
	\hspace{0.3cm} JST, CREST, Sanbancho, Chiyoda-ku, Tokyo, 102-0075, Japan\\
\end{flushleft}
}
\email{shibuta@rikkyo.ac.jp}
\date{}
\begin{document}
\begin{abstract}
We prove that the defining ideal of a sufficiently high Veronese subring of a toric algebra 
admits a quadratic Gr\"obner basis consisting of binomials. 
More generally, we prove that the defining ideal of a sufficiently high Veronese subring 
of a standard graded ring admits a quadratic Gr\"obner basis. 
We give a lower bound on $d$ such that the defining ideal of $d$-th Veronese subring admits a quadratic Gr\"obner basis. 
Eisenbud--Reeves--Totaro stated the same theorem without a proof with some lower bound on $d$. 
In many cases, our lower bound is less than Eisenbud--Reeves--Totaro's lower bound. 
\end{abstract}
\maketitle
\section{Introduction}
We denote by $\N=\{0,1,2,3,\dots\}$ the set of non-negative integers. 
For a multi-index $\a=(a_1,\dots,a_r)\in \N^r$ and variables $\x=x_1,\dots,x_r$, 
we write $\x^\a={x}_1^{a_1}\cdots {x}_r^{a_r}$ and $\abs{\a}=a_1+\dots+a_r$. 
For a given positive integer $d$, we set $\N_d^r=\{\a\in \N^r\mid \abs{\a}=d\}$. 
We denote by $\e_i$ the vector with $1$ in the $i$-th position and zeros elsewhere. 
In this paper, ``quadratic" means ``of degree at most two". 

Let $B=\bigoplus_{i\in \N}B_i$ be a standard $\N$-graded ring (that is, $B$ is generated by $B_1$ over $B_0$ as an algebra) with $B_0=K$ a field. 
For $d\in \N$, we call $B^{(d)}=\bigoplus_{i\in \N}B_{di}$ the {\it $d$-th Veronese subring} of $B$. 
In this paper, we investigate Gr\"obner bases of the defining ideal of $B^{(d)}$. 
We say that a homogeneous ideal admits quadratic Gr\"obner basis 
if there exists a Gr\"obner basis consisting of homogeneous polynomials of degree at most $2$ with respect to some term order. 

We call a finite collection $\A=\{\m^{(1)},\dots,\m^{(s)}\}\subset \Z^n$, $\m^{(i)}=(m^{(i)}_1,\dots,m^{(i)}_n)$,
a {\it configuration} if there exists a vector 
$0\neq \boldsymbol{\lambda}=(\lambda_1,\dots,\lambda_n)\in \Q^n$ 
such that $ \boldsymbol{\lambda}\cdot \m^{(i)}=\sum_j \lambda_j \cdot m^{(i)}_j=1$ for all $i$. 
We denote by $K[\A]$ the standard $\N$-graded $K$-algebra $K[\z^{\m^{(1)}},\dots,\z^{\m^{(s)}}]$. 
For a configuration $\A$, let $\phi_\A$ be the ring homomorphism 
\begin{eqnarray*}
\phi_\A:~ K[{y}_1,\dots, {y}_s]&\to& K[{z}_1^{\pm 1},\dots, {z}_n^{\pm 1}]\\
{y}_i &\mapsto& \z^{\m^{(i)}}=\prod_{j=1}^{n}{z}_j^{m^{(i)}_j}. 
\end{eqnarray*}
We denote $\Ker{\phi_\A}$ by $P_\A$ and call it a {\it toric ideal} of $\A$. 
It is known that the toric ideal $P_\A$ is a homogeneous ideal generated by the binomials $u-v$ 
where $u$ and $v$ are monomials of $K[{y}_1,\dots, {y}_s]$ with $\phi_\A(u)=\phi_\A(v)$. 
We consider the toric ideal $P_{\A^{(d)}}$ 
which is the defining ideal of the $d$-th Veronese subring $K[\A]^{(d)}=K[\A^{(d)}]$ of $K[\A]$ where 
\[
\A^{(d)}=\{a_1\m^{(1)}+\dots+a_s\m^{(s)}\mid (a_1,\dots,a_s)\in\N_d^s \}. 
\]
We will prove the following theorem. 
\begin{theorem0}[Theorem \ref{main toric}]\label{main toric intro}
$P_{\A^{(d)}}$ admits a quadratic Gr\"obner basis for all sufficiently large $d$. 
\end{theorem0}
We prove this theorem in the more general situation: 
Let $S=K[{y}_1,\dots,{y}_s]$ be a standard graded polynomial rings over a field $K$, 
and $I$ a homogeneous ideal of $S$. 
Let $R^{[d]}=K[{x}_\a\mid \a\in \N_d^s]$ be a polynomial ring whose variables 
correspond to the monomials of degree $d$ in $S$, and $\phi_d:R^{[d]}\to S$ a ring homomorphism $\phi_d({x}_\a)=\y^\a$. 
Then $R^{[d]}/\phi_d^{-1}(I)\cong (S/I)^{(d)}$. 
The main result of this paper is the next theorem. 
\begin{theorem0}[Theorem \ref{main}]\label{main intro}
Let $I\subset S$ be a homogeneous ideal, and $\prec$ a term order on $S$. 
Let $\{\y^{\a^{(1)}},\dots,\y^{\a^{(r)}}\}$, 
$\a^{(i)}=(a_{i1}, \dots, a_{is})\in \N^s$, be the minimal system of generators of $\initial_\prec(I)$. 
Then $\phi_d^{-1}(I)$ admits a quadratic Gr\"obner basis if 
\[
d\ge s(\max\{a_{ij}\mid 1\le i\le r, 1\le j\le s\}+1)/2. 
\]
\end{theorem0}
Note that Theorem \ref{main intro} implies Theorem \ref{main toric intro}. 
Eisenbud--Reeves--Totaro \cite{ERT} proved that in the case where the coordinates $y_1,\dots,y_s$ of $S$ are generic, 
$\phi_d^{-1}(I)$ admits quadratic Gr\"obner basis for $d\ge \reg(I)/2$. 
Our lower bound $s(\max\{a_{ij}\mid 1\le i\le r, 1\le j\le s\}+1)/2$ seems large compared with $\reg(I)/2$, but is easy to compute. 
Let $\delta(\initial_\prec(I))=\max\{a_{i1}+\dots+a_{is}\mid 1\le i\le r\}$. 
Eisenbud--Reeves--Totaro gave a easily computable rough lower bound $(s\delta(\initial_\prec(I))-s+1)/2$ 
(with $\prec$ and coordinates $y_1,\dots,y_s$ chosen so that $\delta(\initial_\prec(I))$ is minimal). 
In many cases, our lower bound is less than Eisenbud--Reeves--Totaro's rough lower bound. 
Since coordinate transformation does not preserve the property that an ideal is generated by binomials, 
we can not use generic coordinates to prove Theorem \ref{main toric intro}. 
Our proof does not need any coordinate transformation. 
Eisenbud--Reeves--Totaro stated that the assertion of Theorem \ref{main intro} holds true for $d\ge s\lceil \delta(\initial_\prec(I))/2\rceil$ without a proof (see \cite{ERT} comments after Theorem 11). 
Our lower bound is often less than Eisenbud--Reeves--Totaro's lower bound $s\lceil \delta(\initial_\prec(I))/2\rceil$.

One of the reasons why we are interested in whether a given homogeneous ideal admits quadratic Gr\"obner basis is 
that this is a sufficient condition for the residue ring to be a {\it homogeneous Koszul algebra}. 
We call a graded ring $B$ homogeneous Koszul algebra 
if its residue field has a linear minimal graded free resolution. 
Fr\"oberg proved that if $I$ is generated by monomials of degree two then $S/I$ is Koszul. 
Hence if $I\subset S=K[{y}_1,\dots,{y}_s]$ admits a quadratic initial ideal 
then $B=S/I$ is a homogeneous Koszul algebra by a deformation argument. 
Therefore Theorem \ref{main intro} implies the theorem of Backelin. 
\begin{Theorem}[Backelin \cite{B}]\label{Backelin}
A Veronese subring $B^{(d)}$ of a standard $\N$-graded ring $B=K[B_1]$ over a field $K$ 
is a homogeneous Koszul algebra for all sufficiently large $d\in \N$. 
\end{Theorem}
In the case of $I=0$, Barcanescu and Manolache \cite{BaMa} proved that Veronese subrings of polynomial rings are Koszul. 
We also prove that the defining ideals of Veronese subrings of polynomial rings admit quadratic Gr\"obner bases with respect to 
a certain term order (Theorem \ref{quad GB}). 
See \cite{BGT}, \cite{ERT} and \cite{De Negri} for other term orders that give quadratic Gr\"obner bases. 
\section{Preliminaries on Gr\"obner bases}
Here, we recall the theory of Gr\"obner bases. 
See \cite{Cox1},\cite{Cox2} and \cite{Sturmfels} for details. 

Let $R=K[{x}_1,\dots,{x}_r]$ be a polynomial ring over a field $K$. 
The monomial $\x^\a$ in $R$ is identified with lattice point $\a\in\N^r$. 
A total order $\prec$ on $\N^r$ is a {\em term order} if the zero vector 
$0$ is the unique minimal element, and ${\a }\prec{\b }$ implies 
${\a +\c }\prec {\b +\c }$ for all $\a $, $\b $, $\c\in\N^r$. 
We define $\x^{\a }\prec \x^{\b}$ if ${\a }\prec{\b }$. 
We denote by $\M$ the set of all monomials of $R$. 
Let $\a ={}^{}( a_1,\ldots, a_r)$, 
$\b ={}^{}( b_1,\ldots, b_r)\in \N^r$. 
\begin{Definition}[lexicographic order] 
The {\em lexicographic order} $\prec_{lex}$ with $x_r\prec\dots\prec x_1$ is defined as follows: 
${\a }\prec_{lex} {\b }$ if 
$a _j{<}b _j~\mbox{where}~j=\min\{i\mid a _i\neq b _i\}$. 
\end{Definition}
\begin{Definition}[reverse lexicographic order] 
The {\em reverse lexicographic order} $\prec_{rlex}$ with $x_r\prec\dots\prec x_1$ is defined as follows: 
${\a }\prec_{rlex} {\b }$ if 
$\abs{\a}<\abs{\b}$ or $\abs{\a}=\abs{\b}$ and 
$a _j{>}b _j~\mbox{where}~j=\max\{i\mid a _i\neq b _i\}$. 
\end{Definition}
\begin{Definition}
Let $\prec$ be a term order on $R$, $f\in R$, and $I$ an ideal of $R$. 
The {\it initial term} $\initial_\prec(f)$ is the highest term of $f$ with respect to $\prec$. 
We call $ \initial_\prec(I)=\langle \initial_\prec(f)\mid f\in I\rangle$ the {\it initial ideal} of $I$ with respect to $\prec$. 
We say that a finite subset $G$ of $I$ is a {\it Gr\"obner basis} of $I$ 
with respect to $\prec$ if $\initial_\prec(I) = \langle \initial_\prec(g)\mid g\in G\rangle$. 
\end{Definition}
We give a criterion for a given finite subset of a toric ideal to be a Gr\"obner basis. 
\begin{Lemma}\label{toric GB}
Let $\prec$ be a term order of $R$, $P_\A=\Ker\phi_\A\subset R$ a toric ideal, and $G$ a finite subset of $P_\A$. 
Then $G$ is a Gr\"obner basis of $P_\A$ with respect to $\prec$ if and only if for any monomial $u\not\in \langle \initial_\prec(g)\mid g\in G\rangle$, 
\[
u=\min_\prec\{v\in \M \mid \phi_\A(u)=\phi_\A(v)\}. 
\]
\end{Lemma}
\begin{proof}
Let $u$ be a monomial. 
For a monomial $v$ satisfying $\phi_\A(u)=\phi_\A(v)$, we have $u-v\in P_\A$. 
Therefore $u=\min_\prec\{v\in \M \mid \phi_\A(u)=\phi_\A(v)\}$ if and only if $u\not\in\initial_\prec(P_\A)$. 
Since $G$ is a Gr\"obner basis of $P_\A$ if and only if $u\not\in\initial_\prec(P_\A)$ 
for any monomial $u\not\in \langle \initial_\prec(g)\mid g\in G\rangle$, we conclude the assertion. 
\end{proof}
For a weight vector $\omega=(\omega_1,\dots,\omega_r)\in \N^r$, 
we can grade the ring $R$ by associating weights $\omega_i$ to ${x}_i$. 
To distinguish this grading from the standard one, 
we say that polynomials or ideals of $R$ are {\it $\omega$-homogeneous} if they are homogeneous with respect to the 
graded structure given by $\omega$. 
\begin{Definition}
Given a polynomial $f\in R$ and a weight vector $\omega$, the {\it initial form} $\initial_\omega(f)$ is the sum of all monomials of $f$ 
of the highest weight with respect to $\omega$. 
We call $\initial_\omega(I)=\langle \initial_\omega(f)\mid f\in I\rangle$ the {\it initial ideal} of $I$ with respect to $\omega$. 
If $\initial_\omega(I)$ is a monomial ideal, 
we call $G$ a Gr\"obner basis of $I$ with respect to $\omega$. 
\end{Definition}
We define a new term order constructed from $\omega$ and a term order. 
\begin{Definition}\label{w order}
For a weight vector $\omega$ and a term order $\prec$, we define a new term order $\prec_\omega$ 
constructed from $\omega$ with $\prec$ a tie-breaker 
as following; 
$\x^\a \prec_\omega \x^\b $ if $\omega\cdot\a <\omega\cdot\b $, 
or $\omega\cdot\a =\omega\cdot\b $ and $\x^\a \prec \x^\b $. 
\end{Definition}
We end this section with the following useful propositions about weighted order. See \cite{Sturmfels} for the proofs. 
\begin{Proposition}\label{w}
$\initial_\prec(\initial_\omega(I))=\initial_{\prec_\omega}(I)$. 
\end{Proposition}
\begin{Proposition}
For any term order $\prec$ and any ideal $I\subset R$, 
there exists a weight vector $\omega\in\N^r$ such that $\initial_\prec(I)= \initial_\omega(I)$. 
\end{Proposition}
\section{Proof of the main theorem}
\subsection{Quadratic Gr\"obner bases of $\Ker{\phi_d}$}
Let $S=K[{y}_1,\dots,{y}_s]$ be a standard graded polynomial rings over a field $K$, 
$R^{[d]}=K[{x}_\a\mid \a\in \N_d^s]$ a polynomial ring whose variables 
correspond to the monomials of degree $d$ in $S$, and $\phi_d:R^{[d]}\to S$ the ring homomorphism $\phi_d({x}_\a)=\y^\a$. 
We denote by $\M$ the set of all monomials of $R^{[d]}$. 
In this section, we prove that $\Ker{\phi_d}$ has a quadratic Gr\"obner basis with respect to a certain reverse lexicographic order. 
\begin{Definition}
Let $\prec$ be a term order on $R^{[d]}$. 
For a monomial $u\in R^{[d]}$, we define 
\[
\mv_\prec(u)=\min_\prec\{{x}_\a\in R^{[d]}\mid \y^\a \mbox{~~divides~~} \phi_d(u)\}. 
\]
\end{Definition}
\begin{Lemma}\label{mv divides min}
Let $\prec$ a reverse lexicographic order on $R^{[d]}$, and $u\in R^{[d]}$ a monomial. Then 
\[
\min_\prec\{v\in\M \mid \phi_d(u)=\phi_d(v)\}
=\mv_\prec(u)\cdot \min_\prec\{v'\in \M \mid \frac{\phi_d(u)}{\phi_d(\mv_\prec(u))}=\phi_d(v')\}. 
\]
\end{Lemma}
\begin{proof}
Since $\phi_d(u)/\phi_d(\mv_\prec(u))$ is a monomial whose degree is divisible by $d$, 
there exists a monomial $v'\in\M$ such that $\phi_d(v')=\phi_d(u)/\phi_d(\mv_\prec(u))$. 
Let $u_0=\min_\prec\{v\in\M \mid \phi_d(u)=\phi_d(v)\}$. 
Since $\prec$ is a reverse lexicographic order and $u_0\prec \mv_\prec(u)\cdot v'$, $u_0$ is divided by $\mv_\prec(u)$. 
Hence the assertion follows. 
\end{proof}
We will give a criterion for a finite subset of $\Ker{\phi_d}$ to be a Gr\"obner basis of $\Ker{\phi_d}$ 
with respect to a reverse lexicographic order. 
\begin{Proposition}\label{rev lex criterion}
Let $\prec$ be a reverse lexicographic order on $R^{[d]}$, and $G$ a finite subset of $\Ker{\phi_d}$. 
Then $G$ is a Gr\"obner basis of $\Ker{\phi_d}$ with respect to $\prec$ 
if and only if $\mv_\prec(u)$ divides $u$ for any monomial $u\not\in \langle \initial_\prec(g)\mid g\in G\rangle$. 
\end{Proposition}
\begin{proof}
Suppose that $G$ is a Gr\"obner basis, and let $u\not\in \langle \initial_\prec(g)\mid g\in G\rangle=\initial_\prec(I)$ 
be a monomial. 
Then $\mv_\prec(u)$ divides $u$ by Lemma \ref{toric GB} and Lemma \ref{mv divides min}. 

Conversely, suppose that $\mv_\prec(u)$ divides $u$ for any monomial $u\not\in \langle \initial_\prec(g)\mid g\in G\rangle$. 
We will prove $u=\min_\prec\{v\in \M \mid \phi_d(u)=\phi_d(v)\}$ by induction on the degree of $u$. 
Since $u/\mv_\prec(u)$ is also not in $\langle \initial_\prec(g)\mid g\in G\rangle$, it holds that 
\[
u/\mv_\prec(u)=\min_\prec\{v\in \M \mid \phi_d(u/\mv_\prec(u))=\phi_d(v)\} 
\]
by the assumption of induction. 
By Lemma \ref{mv divides min}, we conclude 
\begin{eqnarray*}
u&=&\mv_\prec(u)\cdot (u/\mv_\prec(u))=\mv_\prec(u)\cdot \min_\prec\{v\in \M \mid \phi_d(u/\mv_\prec(u))=\phi_d(v)\}\\
 &=&\min_\prec\{v\in \M \mid \phi_d(u)=\phi_d(v)\}. 
\end{eqnarray*}
Hence $G$ is a Gr\"obner basis of $I$ by Lemma \ref{toric GB}. 
\end{proof}
\begin{Definition}
Let $\a={}^{}(a_1,\dots,a_s)\in \N^s$, and $\sigma \in\mathfrak{S}_s$ be a permutation of indices such that 
$a_{\sigma(1)}\le a_{\sigma(2)}\le \dots \le a_{\sigma(s)}$. We define $\Gamma(\a)\in \N^s$ to be 
\[
\Gamma(\a)=(a_{\sigma(1)}, a_{\sigma(2)}, \dots, a_{\sigma(s)}). 
\]
\end{Definition}
\begin{Definition}
Let $\prec_\Gamma$ be a reverse lexicographic order on $R^{[d]}$ such that the order on variables is defined as follows: 
${x}_\a\prec_\Gamma {x}_\b$ if $\Gamma(\b)\prec_{lex} \Gamma(\a)$ or $\Gamma(\b)= \Gamma(\a)$ and $\b\prec_{lex} \a$. 
\end{Definition}
\begin{Example}
In the case of $s=2$ and $d=4$, 
\[
x_{(2,2)}\prec_\Gamma x_{(3,1)}\prec_\Gamma x_{(1,3)}\prec_\Gamma x_{(4,0)}\prec_\Gamma x_{(0,4)}. 
\]
In the case of $s=3$ and $d=3$, 
\begin{eqnarray*}
&&x_{(1,1,1)}\prec_\Gamma x_{(2,1,0)}\prec_\Gamma x_{(2,0,1)}\prec_\Gamma x_{(1,2,0)}
\prec_\Gamma x_{(1,0,2)}\prec_\Gamma x_{(0,2,1)}\prec_\Gamma\\
&&x_{(0,1,2)}\prec_\Gamma x_{(3,0,0)}\prec_\Gamma x_{(0,3,0)}\prec_\Gamma x_{(0,0,3)}. 
\end{eqnarray*}
\end{Example}
The following are some typical and important properties of the term order $\prec_\Gamma$. 
\begin{Lemma}\label{remarks on order}
Let $\a={}^{}(a_1,\dots,a_s), \b={}^{}(b_1,\dots,b_s)\in \N_d^s$. 
\begin{enumerate}
\item $\Gamma(\a)$ is the minimal element of $\{(a_{\sigma(1)}, a_{\sigma(2)}, \dots, a_{\sigma(s)})\mid \sigma \in \mathfrak{S}_s\}$ with 
respect to the lexicographic order. 
\item If $\#\{i\mid a_i\neq 0\}>\#\{i\mid b_i\neq 0\}$ where $\# F$ is the cardinality of the set $F$, then $\a\prec_\Gamma \b$. 
\item Suppose that $a_j-a_i\ge 2$ for some $1\le i,j \le s$, and let $\a'=\a+\e_i-\e_j$. 
Then $x_{\a'}\prec_\Gamma x_{\a}$. 
\item Suppose that $a_j-a_i=1$ for some $1\le i,j \le s$, and let $\a'=\a+\e_i-\e_j$. 
Then $x_{\a'}\prec_\Gamma x_{\a}$ if and only if $i<j$. 
\item Let $u\in R^{[d]}$ be a monomial, and suppose that ${x}_\a=\mv_{\prec_\Gamma}(u)$. 
If $a_j-a_i\ge 2$ for some $1\le i,j \le s$, then the degree of $\phi_d(u)$ in the variable ${y}_i$ is $a_i$. 
\end{enumerate}
\end{Lemma}
\begin{proof}
The assertions of (1), (2), and (3) follow immediately from the definition of the term order $\prec_\Gamma$. 

(4) Note that $\Gamma(\a)=\Gamma(\a')$, and $\a'$ is the vector obtained from $\a$ by swapping the $i$-th and $j$-th components. 
Hence $x_{\a'}\prec_\Gamma x_{\a}$ if and only if $\a\prec_{lex} \a'$, which is equivalent to $i<j$. 
 
(5) Assume, to the contrary, that the degree of $\phi_d(u)$ in the variable ${y}_i$ is strictly greater than $a_i$. 
Then $\y^{\a'}$ divides $\phi_d(u)$ where $\a'=\a+\e_i-\e_j$. 
By (3), we have $x_{\a'}\prec_\Gamma x_{\a}$ which contradicts the definition of $\mv_{\prec_\Gamma}(u)$. 
\end{proof}
\begin{Theorem}\label{quad GB}
Let 
\[
G_\Gamma=\{{x}_{\a+\e_i}{x}_{\b+\e_j}-{x}_{\a+\e_j}{x}_{\b+\e_i}\mid \a,\b\in \N^s_{d-1},~1\le i<j\le s \}. 
\]
Then $G_\Gamma$ is a Gr\"obner basis of $\Ker{\phi_d}$ with respect to $\prec_\Gamma$. 
\end{Theorem}
\begin{proof}
As $(\a+\e_i)+(\b+\e_j)=(\a+\e_j)+(\b+\e_i)$ for $\a,~\b\in\N_{d-1}^s$, $G_\Gamma$ is a finite subset of $\Ker{\phi_d}$. 
Let $u\in R^{[d]}$ be a monomial such that ${x}_\a=\mv_{\prec_\Gamma}(u)$ does not divide $u$.  
To conclude the assertion, it is enough to show that $u\in \langle \initial_{\prec_\Gamma}(g)\mid g\in G_\Gamma\rangle$ by Proposition \ref{rev lex criterion}. 
Take $\a={}^{}(a_1,\dots,a_s),~\b={}^{}(b_1,\dots,b_s)\in \N_d^s$ such that 
\begin{eqnarray*}
{x}_\a&=&\mv_{\prec_\Gamma}(u), \\
{x}_\b&=&\min_{\prec_\Gamma}\{{x}_\c\in R^{[d]}\mid {x}_\c \mbox{~~divides~~} u\}. 
\end{eqnarray*}
Note that $\a\neq \b$ and ${x}_\a\prec_\Gamma {x}_\b$. 
Let $\tau\in \mathfrak{S}_s$ be a permutation of indices 
such that $b_{\tau(1)}\le \dots \le b_{\tau(s)}$. 
Since \[
(b_{\tau(1)},b_{\tau(2)},\dots,b_{\tau(s)})=\Gamma(\b)\preceq_{lex}\Gamma(\a)\preceq_{lex}(a_{\tau(1)},a_{\tau(2)},\dots,a_{\tau(s)}) 
\]
by Lemma \ref{remarks on order} (1), 
and $(b_{\tau(1)},\dots,b_{\tau(s)})\neq (a_{\tau(1)},\dots,a_{\tau(s)})$, 
there exists $1\le {j}\le s$ such that $a_{\tau(i)}=b_{\tau(i)}$ for all $i<{j}$ and 
\[
b_{\tau({j})}<a_{\tau({j})}. 
\]
As $\abs{\a}=\abs{\b}=d$, there exists ${k}>{j}$ such that 
\[
b_{\tau({k})}>a_{\tau({k})}. 
\]
As $\y^\b$ divides $\phi_d(u)$ and $b_{\tau({k})}>a_{\tau({k})}$, 
we have 
\[
a_{\tau(j)}-a_{\tau(k)}\le 1 
\]
by Lemma \ref{remarks on order} (5). 
As 
\[
a_{\tau(j)}-a_{\tau(k)}=(a_{\tau({j})}-b_{\tau({j})})+(b_{\tau({j})}-b_{\tau(k)})+(b_{\tau(k)}-a_{\tau({k})}), 
\]
and $a_{\tau({j})}-b_{\tau({j})},~b_{\tau(k)}-a_{\tau({k})}>0$, we have 
\[
b_{\tau(k)}-b_{\tau(j)}>0. 
\]
Since $\y^\a$ divides $\phi_d(u)$, the degree of $\phi_d(u/{x}_\b)$ in the variable ${y}_{\tau(j)}$ is not less than $a_{\tau({j})}-b_{\tau({j})}>0$, 
and thus there exists $\c={}^{}(c_1,\dots,c_s)\in \N_d^s$ such that $c_{\tau({j})}>0$ and ${x}_\c$ divides $u/{x}_\b$. 
We set 
\begin{eqnarray*}
\b'&=&\b+\e_{\tau({j})}-\e_{\tau({k})}, \\
\c'&=&\c-\e_{\tau({j})}+\e_{\tau({k})}. 
\end{eqnarray*}
Then ${x}_\b {x}_\c-{x}_{\b'}{x}_{\c'}\in G_\Gamma$ and ${x}_\b {x}_\c$ divides $u$. 
To complete the proof, we will show that ${x}_\b {x}_\c$ is the initial term of ${x}_\b {x}_\c-{x}_{\b'}{x}_{\c'}$. 
Since ${x}_\b\prec_\Gamma {x}_\c$ and $\prec_\Gamma$ is a reverse lexicographic order, it is enough to show that ${x}_{\b'}\prec_\Gamma {x}_\b$. 
In the case of $b_{\tau({k})}-b_{\tau({j})}\ge 2$, we have ${x}_{\b'}\prec_\Gamma {x}_\b$ by Lemma \ref{remarks on order} (3). 
In the case of $b_{\tau({k})}-b_{\tau({j})}=1$, 
we have 
\[
a_{\tau({j})}-a_{\tau({k})}=(a_{\tau({j})}-b_{\tau({j})})-1+(b_{\tau(k)}-a_{\tau({k})})\ge 1, 
\]
and hence $a_{\tau({j})}-a_{\tau({k})}=1$. 
Let $\a'=\a-\e_{\tau({j})}+\e_{\tau({k})}$. 
Then $\y^{\a'}$ divides $\phi_d(u)$ as $b_{\tau({k})}>a_{\tau({k})}$. Thus ${x}_\a\prec_\Gamma {x}_{\a'}$ 
by the definition of $\mv_{\prec_\Gamma}(u)$. 
This implies ${\tau({j})}<{\tau({k})}$ by Lemma \ref{remarks on order} (4). 
Therefore ${x}_{\b'} \prec_\Gamma {x}_\b$ again by Lemma \ref{remarks on order} (4). 
\end{proof}
See \cite{BGT}, \cite{ERT} and \cite{De Negri} for other term orders that give quadratic Gr\"obner bases of $\Ker\phi_d$. 

We already proved Theorem \ref{main intro} in the case of $I=0$. 
In the rest of this paper, we prove that there exists a term order on $R^{[d]}$ such that 
the initial ideal of $\phi_d^{-1}(I)$ is generated by quadratic monomials for all sufficiently large $d$ 
for any homogeneous ideal $I\subset S$. 
First, we will prove this in the case where $I$ is a monomial ideal, and then reduce the general case to the monomial ideal case. 
\subsection{In the case of monomial ideals}
\begin{Definition}\label{def of L(I)}
Let $I\subset S$ be a monomial ideal, and $\prec$ any term order on $R^{[d]}$. 
We define 
\[
L_{\prec}(I)=\langle \M \cap (\phi_d^{-1}(I)\backslash \initial_\prec(\Ker{\phi_d})) \rangle
\]
to be the monomial ideal of $R^{[d]}$ generated by all monomials in $\phi_d^{-1}(I)\backslash \initial_\prec(\Ker{\phi_d})$, 
and $M_{\prec}(I)$ to be the minimal system of generators of $L_{\prec}(I)$ consisting of monomials. 
\end{Definition}
\begin{Lemma}\label{monomial cases}
Let $I\subset S$ be a monomial ideal, and $\prec$ any term order on $R^{[d]}$. 
Let $G$ be a Gr\"obner basis of $\Ker{\phi_d}$ with respect to $\prec$. 
Then $G\cup M_{\prec}(I)$ is a Gr\"obner basis of $\phi_d^{-1}(I)$ with respect to $\prec$. 
\end{Lemma}
\begin{proof}
First, we note that $G\cup M_{\prec}(I) \subset \phi_d^{-1}(I)$. 
Take $f\in \phi_d^{-1}(I)$, and let $g$ be the remainder on division of $f$ by $G$. 
Then any term of $g$ is not in $\initial_\prec(\Ker{\phi_d})$. 
Hence different monomials appearing in $g$ map to different monomials under $\phi_d$. 
Since $I$ is a monomial ideal, it follows that all terms of $g$ are in $L_{\prec}(I)$. 
Thus the remainder on division of $g$ by $M_{\prec}(I)$ is zero. 
Therefore a remainder on division of $f$ by $G\cup M_{\prec}(I)$ is zero. 
This implies that $G\cup M_{\prec}(I)$ is a Gr\"obner basis of $\phi_d^{-1}(I)$. 
\end{proof}
\begin{Proposition}\label{mv in}
Let $\a={}^{}(a_1,\dots,a_s)\in \N^s$, and $a=\max\{a_i\mid i=1,\dots,s\}$. 
Assume that $d\ge s(a+1)/2$. 
Let $u\in (\phi_d^{-1}(I)\backslash \initial_\prec(\Ker{\phi_d})$ be a monomial of degree $\ge 2$, 
and set ${x}_\b=\mv_{\prec_\Gamma}(u)$, $\b=(b_1,\dots,b_s)\in \N_d^s$ $x_\c=\mv_{\prec_\Gamma}(u/x_\b)$, $\c=(c_1,\dots,c_s)\in \N_d^s$. 
Then $x_\b x_\c\in \phi_d^{-1}(\y^\a)$. 
In particular, $L_{\prec_\Gamma}(\y^\a)$ is generated by quadratic monomials. 
\end{Proposition}
\begin{proof}
First, note that $x_\c$ is well-defined since $x_\b$ divides $u$ by Lemma \ref{mv divides min}. 
Assume, to the contrary, that ${x}_\b x_\c\not\in \phi_d^{-1}(\y^\a)$. 
Since $\y^\a$ does not divide $\y^{\b+\c}=\phi_d(x_\b x_\c)$, we have $b_i+c_i<a_i$ for some $1\le i \le s$. 
On the other hand, since $\abs{\b+\c}=2d\ge s(a+1)$, there exists $1\le j \le s$ ($j\neq i$) such that $b_j+c_j>a+1$. 
Hence $(b_j+c_j)-(b_i+c_i)=(b_j+c_j-a)+(a-b_i+c_i)\ge (b_j+c_j-a)+(a_i-b_i+c_i) \ge 3$. 
Thus we have $b_j-b_i\ge 2$ or $c_j-c_i\ge 2$. 
Since $\y^\a$ divides $\phi_d(u)$ and $b_i+c_i<a_i$, 
the degree of $\phi_d(u)$ in the variable $y_i$ is strictly greater than $b_i$, 
and the degree of $\phi_d(u/x_\b)$ in the variable $y_i$ is strictly greater than $c_i$. 
This contradicts to Lemma \ref{remarks on order} (5). Hence ${x}_\b x_\c\in \phi_d^{-1}(\y^\a)$. 
Since $x_\b x_\c$ divides $u$ if $u\neq x_\b$ by Lemma \ref{mv divides min}, 
$L_{\prec_\Gamma}(\y^\a)$ is generated by quadratic monomials. 
\end{proof}
\begin{Theorem}\label{main monomial}
Let $I\subset S$ be a monomial ideal with a system of generators $\{\y^{\a^{(1)}},\dots,\y^{\a^{(r)}}\}$, 
$\a^{(i)}=(a_{i1}, \dots, a_{is})\in \N^s$, 
and set 
\[
a=\max\{a_{ij}\mid 1\le i\le r,~1\le j\le s \}. 
\]
If $d\ge s(a+1)/2$, then $\initial_{\prec_\Gamma}(\phi_d^{-1}(I))$ is generated by quadratic monomials. 
\end{Theorem}
\begin{proof}
Let $u\in \phi_d^{-1}(I)$ be a monomial. Then $u\in \phi_d^{-1}(\y^{\a^{(i)}})$ for some $i$. 
Thus it follows that $L_{\prec_\Gamma}(I)=\sum_{i=1}^r L_{\prec_\Gamma}(\y^{\a^{(i)}})$. 
Hence $M_{\prec_\Gamma}(I)$ consists of quadratic monomials by Proposition \ref{mv in}. 
By Lemma \ref{monomial cases}, the Gr\"obner basis of $\phi_d^{-1}(I)$ with respect to $\prec_\Gamma$ is 
the union of $M_{\prec_\Gamma}(I)$ and $G_\Gamma$ in Theorem \ref{quad GB}. 
This proves the assertion. 
\end{proof}
\subsection{In the case of homogeneous ideals}
Let $I\subset S$ be a homogeneous ideal, 
and fix a weight vector $\omega$ of $S$ such that $\initial_\omega(I)$ is a monomial ideal. 
We denote by $\phi_d^*\omega $ the weight vector on $R^{[d]}$ that assign $\omega\cdot \a$ to the weight of ${x}_\a$ for $\a\in \N_d^s$. 
Since the weight of ${x}_\a$ coincides with the weight of $\y^\a=\phi_d({x}_\a)$, 
$\phi_d$ is a homogeneous homomorphism of degree zero with respect to the graded ring structures on $R^{[d]}$ and $S$ 
defined by $\phi_d^*\omega $ and $\omega$. 
For the simplicity of the notation, 
we regard zero-polynomial as a homogeneous ($\omega$-homogeneous) polynomial of any degree (weight). 
\begin{Lemma}\label{pull back of initial}
With the notation as above, $\initial_{\phi_d^*\omega }(\phi_d^{-1}(I))=\phi_d^{-1}(\initial_\omega(I))$. 
\end{Lemma}
\begin{proof}
First, 
note that $\initial_{\phi_d^*\omega }(\phi_d^{-1}(I))$ and $\phi_d^{-1}(\initial_\omega(I))$ are both homogeneous and $\phi_d^*\omega $-homogeneous. 
Since $\phi_d$ sends $\phi_d^*\omega $-homogeneous polynomials to $\omega$-homogeneous polynomials, 
we have $\phi_d(\initial_{\phi_d^*\omega }(g))=\initial_\omega(\phi_d(g))$ for all $g\in R^{[d]}$. 
Hence it follows that $\initial_{\phi_d^*\omega }(\phi_d^{-1}(I))\subset \phi_d^{-1}(\initial_\omega(I))$. 

We will prove the converse inclusion. 
Let $\{f_1,\dots,f_r\}$ be a Gr\"obner basis of $I$ with respect to $\omega$ consisting of homogeneous polynomials. 
Let $g\in \phi_d^{-1}(\initial_\omega(I))$ be a homogeneous and $\phi_d^*\omega $-homogeneous polynomial. 
We set $\ell$ and $m$ to be the degree and the weight of $g$. 
Then $\phi_d(g)$ is a homogeneous and $\omega $-homogeneous polynomial of degree $d\ell$ and of weight $m$. 
Since $\phi_d(g)\in \initial_\omega(I)$, 
there exist homogeneous and $\omega$-homogeneous polynomials $h_1,\dots,h_r$ 
such that 
$\phi_d(g)=\sum_{i=1}^{r} h_i \cdot \initial_\omega(f_i)$,  
and $h_i \cdot\initial_\omega(f_i)$ is of degree $d\ell$ and of weight $m$. 
We set $q=\sum_{i=1}^{r} h_i f_i$. 
Then $q$ is a homogeneous polynomial of degree $d\ell$ satisfying $q\in I$, and 
$\initial_\omega(q)=\sum_{i=1}^{r} h_i \cdot \initial_\omega(f_i)=\phi_d(g)$. 
We write $q=\initial_\omega(q)+\sum_{i<m}q_i$ where $q_i$ is a homogeneous and $\phi_d^*\omega $-homogeneous polynomial 
of degree $d\ell$ and of weight $i$. 
For $i<m$, there exists $g_i\in R^{[d]}$ a homogeneous and $\omega$-homogeneous polynomial of degree $\ell$ and of weight $i$ 
such that $\phi_d(g_i)=q_i$. 
Then $\phi_d(g+\sum_{i<m}g_i)=q$ and $\initial_{\phi_d^*\omega }(g+\sum_{i<m}g_i)=g$. 
Therefore we have $g\in \initial_{\phi_d^*\omega }(\phi_d^{-1}(I))$. 
\end{proof}
Now, we are ready to prove the main theorem of this paper. 
\begin{Theorem}\label{main}
Let $I\subset R^{[d]}$ be a homogeneous ideal, 
and fix a weight vector $\omega$ of $S$ such that $\initial_\omega(I)$ is a monomial ideal. 
Let $\{\y^{\a^{(1)}},\dots,\y^{\a^{(r)}}\}$, 
$\a^{(i)}=(a_{i1}, \dots, a_{is})\in \N^s$, be the minimal system of generators of $\initial_\omega(I)$ and set 
\[
a=\max\{a_{ij}\mid 1\le i\le r,~1\le j\le s \}. 
\] 
Let $\prec_{\Gamma_\omega}$ be the term order on $R^{[d]}$ constructed from $\phi_d^*\omega $ with $\prec_\Gamma$ a tie-breaker as in Definition \ref{w order}. 
Then $\initial_{\prec_{\Gamma_\omega}}(\phi_d^{-1}(I))$ is generated by quadratic monomials if $d\ge s(a+1)/2$. 
\end{Theorem}
\begin{proof}
By Proposition \ref{w} and Lemma \ref{pull back of initial}, we have 
\[
\initial_{\prec_{\Gamma_\omega}}(\phi_d^{-1}(I))=\initial_{\prec_\Gamma}(\initial_{\phi_d^*\omega }(\phi_d^{-1}(I)))
=\initial_{\prec_\Gamma}(\phi_d^{-1}(\initial_{\omega}(I))). 
\]
Since $\initial_{\omega}(I)$ is a monomial ideal, the assertion follows from Theorem \ref{main monomial}
\end{proof}
\begin{Observation}
Let the notation be as in Theorem \ref{main}. 
We will compare our lower bound on $d$ with Eisenbud--Reeves--Totaro's lower bound. 
We set $\delta(\initial_\omega(I))=\max\{a_{i1}+\dots+a_{is}\mid 1\le i\le r\}$. 

Eisenbud--Reeves--Totaro \cite{ERT} proved that $\phi_d^{-1}(I)$ has quadratic initial ideal 
for $d\ge \reg(I)/2$ in the case where the coordinates $y_1,\dots,y_s$ of $S$ are generic. 
Our lower bound $s(a+1)/2$ seems large compared with $\reg(I)/2$, but is easy to compute. 
Eisenbud--Reeves--Totaro also gave a easily computable rough lower bound $(s\delta(\initial_\omega(I))-s+1)/2$. 
Our lower bound is less than Eisenbud--Reeves--Totaro's rough lower bound if and only if $a+2\le \delta(\initial_\omega(I))$. 
Thus there exist a lot of examples in which our lower bound is less than Eisenbud--Reeves--Totaro's rough lower bound. 

If the coefficient field $K$ is finite, or we are interested in Gr\"obner bases consisting of binomials, 
we can not deal with generic coordinates.  
Eisenbud--Reeves--Totaro stated without a proof that $\phi_d^{-1}(I)$ has quadratic initial ideal 
if $d\ge s\lceil \delta(\initial_\omega(I))/2\rceil$ 
(with $\prec$ and coordinates $y_1,\dots,y_s$ chosen so that $\delta(\initial_\prec(I))$ is minimal, 
see \cite{ERT} comments after Theorem 11). 
If $\delta(\initial_\prec(I))$ is odd, our lower bound is always not greater than Eisenbud--Reeves--Totaro's lower bound 
$s\lceil \delta(\initial_\prec(I))/2\rceil$. 
If $\delta(\initial_\prec(I))$ is even, our lower bound is greater than Eisenbud--Reeves--Totaro's lower bound 
only in the case where the inequality $a\ge \delta(\initial_\omega(I))$ holds. 
This inequality holds if and only if there exist $1\le i\le r$, $1\le j\le s$ and $N\in \N$ such that 
$\y^{\a^{(i)}}=y_j^N$ and $\deg \y^{\a^{(k)}}\le N$ for all $1\le k \le r$ which does not occur very often. 
\end{Observation}
Applying Theorem \ref{main} to toric ideals, we obtain the next theorem. 
\begin{Theorem}\label{main toric}
With the notation as in the introduction, $P_{\A^{(d)}}$ admits a quadratic Gr\"obner basis for sufficiently large $d$. 
\end{Theorem}
\begin{Remark}
It is easy to show that if $I$ admits a squarefree initial ideal, 
then $\phi_d^{-1}(I)$ also admits a squarefree initial ideal using Lemma \ref{monomial cases} and Lemma \ref{pull back of initial} 
(the lexicographic order in \cite{De Negri} gives squarefree initial ideal of $\Ker{\phi_d}$, and 
if $I$ is a squarefree monomial ideal then so is $L_{\prec}(I)$ in Definition \ref{def of L(I)} for any term order $\prec$). 
However, $\initial_{\prec_\Gamma}(\Ker{\phi_d})$ is not squarefree, and it seems to be an open question 
whether $\phi_d^{-1}(I)$ admits a quadratic squarefree initial ideal if $I$ has a squarefree initial ideal. 
\end{Remark}

\end{document}